\documentclass[reqno]{amsart}

\usepackage{amsmath,amsthm} 
\usepackage{amssymb,mathrsfs} 
\usepackage{graphicx} 
\usepackage{color,subfigure} 
\usepackage{enumerate}  
\usepackage{times}

\newtheorem{theorem}{Theorem}
\newtheorem{lemma}{Lemma}
\newtheorem{proposition}{Proposition}
\newtheorem{corollary}{Corollary}
\newtheorem{remark}[theorem]{Remark}

\DeclareMathAlphabet{\mathpzc}{OT1}{pzc}{m}{it}
\newcommand{\dps}{\displaystyle } 
\newcommand{\rme}{\mathrm{e}} 
\newcommand{\cL}{\mathcal{L}}
\newcommand{\cA}{\mathcal{A}}
\newcommand{\sL}{\mathscr{L}}

\newcommand{\cD}{\mathcal{D}}
\newcommand{\wcL}{\widetilde{\mathcal{L}}}

\renewcommand{\leq}{\leqslant}
\renewcommand{\geq}{\geqslant}

\newcommand{\bone}{\mathbf{1}}
\newcommand{\Ak}{A}
\newcommand{\Uk}{U}

\newcommand{\cB}{\mathcal{B}}
\newcommand{\nuref}{\nu_{\rm ref}}
\newcommand{\onuref}{\bar\nu_{\mathrm{ref}}}
\newcommand{\nurefz}{\nu_{\mathrm{ref},z}}
\newcommand{\R}{\mathbb{R}}

\newcommand{\cE}{\mathcal{E}}
\newcommand{\llll}{\left\langle\left\langle}
\newcommand{\rrrr}{\right\rangle\right\rangle}
\newcommand{\fh}{\mathfrak{h}}
\newcommand{\wfh}{\widetilde{\mathfrak{h}}}

\DeclareMathOperator{\var}{var}

\begin{document}

  
\title[Longtime convergence of TAMD]{Longtime convergence of the 
  Temperature-Accelerated Molecular Dynamics Method}

\author{Gabriel Stoltz}
\address{Universit\'e Paris-Est, CERMICS (ENPC), INRIA, F-77455 Marne-la-Vall\'ee, France}
\email{gabriel.stoltz@enpc.fr}

\author{Eric Vanden-Eijnden}
\address{Courant Institute of Mathematical Sciences, New York University, 251 Mercer street, New York, NY 10012, USA}
\email{eve2@cims.nyu.edu}


\numberwithin{equation}{section}
 
\maketitle

\begin{abstract}
  The equations of the temperature-accelerated molecular dynamics
  (TAMD) method for the calculations of free energies and partition
  functions are analyzed. Specifically, the exponential convergence of
  the law of these stochastic processes is established, with a convergence rate close to the one
  of the limiting, effective dynamics at higher temperature obtained with infinite acceleration. It is
  also shown that the invariant measures of TAMD are close to a known reference
  measure, with an error that can be quantified precisely. Finally, a
  Central Limit Theorem is proven, which allows the estimation of errors
  on properties calculated by ergodic time averages. 
  These results not only demonstrate the usefulness and
  validity range of the TAMD equations, but they also permit in
  principle to adjust the parameter in these equations to optimize
  their efficiency.
\end{abstract}

\section{Introduction}
\label{sec:intro}

The estimation of free energies and partition functions is a classical
problem in statistical
physics~\cite{landau1980statistical,feynman1998statistical,Balian}.
Given a potential $U: \cD_q\times \cD_z \to \R$ defined on some domain
$\cD_q \times \cD_z \subset \R^d \times \R^n$, it amounts to
calculating
\begin{equation}
  \label{eq:1}
  A(z) = -\beta^{-1} \log Z(z), \qquad Z(z) = \int _{\cD_q} \rme^{-\beta U(q,z) } dq
\end{equation}
where $\beta>0$ is a parameter and it is assumed that the properties
of~$U$ are such that the integral in~\eqref{eq:1} converges. The free
energy $A(z)$ and the partition function $Z(z)$ associated with $U$ for
a given value $\beta$ are informative in a variety of contexts. For example:

\begin{itemize}
\item In standard statistical mechanics, where the measure
  $\rme^{-\beta U(q,z)} dq$ properly normalized is the equilibrium
  measure of a thermal system at inverse temperature $\beta$
  (canonical ensemble), $A(z)$ and $Z(z)$ give important information about
  the thermodynamic properties of this system as a function of the
  control parameter $z$.
\item The calculation of $A(z)$ and $Z(z)$ is also useful in machine
  learning, where many generative models are defined in terms of
  unnormalized probability measures. In this set-up, the calculation
  of $Z(z)$ allows the estimation of the probability of a state $q$ in the
  data set, given the parameter $z$.
\item Let
  \begin{equation*}
    U(q,z) = V(q)+ k\|\xi(q) -z\|^2
  \end{equation*}
  where $V: \cD_q\to \R$ is the potential energy function, $k>0$ is a
  parameter, $\|\cdot\|$ some norm on~$\R^n$, and
  $\xi : \cD_q\to \cD_z$ some collective variables giving a
  coarse-grained information on the system. In this setting, $A(z)$
  and $Z(z)$ give information about the push forward of the measure
  proportional to $\rme^{-\beta V(q)}dq$ by the map $\xi$ (also known
  as the image measure by~$\xi$).  Specifically, as $k\to\infty$,
  $Z(z)$ converges to the density associated with this pushforward
  measure.
\end{itemize}

Many methods have been introduced to calculate $A(z)$ and
$Z(z)$. Since one is typically interested in situations where the
dimensionality of the domain $\cD_q$ is large, $d\gg1$, a calculation
of the integral in~\eqref{eq:1} by analytical tools or standard
numerical integration methods is usually hopeless. In these
situations, one therefore typically resorts to methods aiming at
sampling the measure proportional to $\rme^{-\beta U(q,z)} \, dq \, dz$, upon
noticing that $Z(z)$ is the density of the marginal of this measure in
the variables~$z$ . Alternatively, the gradient of the free energy
$A(z)$ can be expressed as an expectation with respect to the measure
proportional to $\rme^{-\beta U(q,z)}dq$ as
\begin{equation}
  \label{eq:2}
  \nabla_z A(z) = \frac{\dps \int _{\cD_q} \nabla_z U(q,z) \, \rme^{-\beta U(q,z) } \, dq}
        {\dps \int _{\cD_q} \rme^{-\beta U(q,z) } \, dq},
\end{equation}
from which $A(z)$ can then be estimated by so-called thermodynamic
integration~\cite{kirkwood-35}.  The sampling of $Z(z)$ or the
calculation of the expectation can be done e.g. via
Metropolis-Hastings Monte-Carlo sampling~\cite{MRRTT53,Hastings70} or
via time-averaging along the solutions of certain stochastic
differential equations (SDE), see for instance the review~\cite{LS16}.
Possible SDEs are
\begin{equation}
  \label{eq:overdamped}
  dq_t = -\nabla_q \Uk(q_t,z)\, dt 
  + \sqrt{2\beta^{-1}} \, dW_t, 
\end{equation}
or
\begin{equation}
  \label{eq:inertial}
  \left\{ \begin{aligned}
      dq_t & = M^{-1} p_t \, dt, \\
      dp_t & = - \nabla_q \Uk(q_t,z)\, dt - \gamma\,
      M^{-1}p_t \, dt 
      + \sqrt{2\gamma\beta^{-1}} \, dW_t,
  \end{aligned} \right.
\end{equation}
where $M$ is a positive definite matrix called mass matrix, $\gamma>0$ is the
friction coefficient, and $W_t$ is a standard $d$-dimensional Wiener
process.  In the chemical-physics literature,  \eqref{eq:overdamped}
and \eqref{eq:inertial} are referred to as the overdamped and
inertial Langevin equations, respectively. Under suitable assumptions
on~$U$ and the nature of the domain~$\cD_q$, both these equations are
ergodic with respect to equilibrium measures whose
configurational part is proportional to $\rme^{-\beta U(q,z)} dq$,
implying that
\begin{equation}
  \label{eq:4}
  \nabla_z A(z) =  \lim_{T\to\infty} \frac1T \int_0^T \nabla_z U(q_t,z) \, dt.
\end{equation}
In practice, however, vanilla approaches based on estimating
$\nabla_zA(z)$ via~\eqref{eq:4} typically fail. The reason is that the
evolution of~$q_t$ in $\cD_q$ can be impeded by barriers of energetic
or entropic origin created by the complex structure of~$U$. These
barriers create dynamical bottlenecks that may render the
convergence of the time-average at the right-hand side of~\eqref{eq:4}
very slow on the clock over which the numerical integration
of~\eqref{eq:overdamped} or~\eqref{eq:inertial} need to be
performed.

Several methods have been introduced to remedy this problem (see
e.g.~\cite{chipot-pohorille-07,abrams2014enhanced,
  pietrucci2017strategies} for recent reviews
and~\cite{stoltz2010free} for a mathematical perspective on the
topic). As a rule, these methods are based on some appropriate
modifications of the evolution equations in~\eqref{eq:overdamped}
and~\eqref{eq:inertial} that allow the trajectory to surmount the
barriers in~$U$ and explore $\cD_q$ faster, thereby accelerating the
convergence of~\eqref{eq:4}.

The main objective of this paper is to analyze one such method, namely
the temperature-accelerated molecular dynamics (TAMD)
technique~\cite{MVE06,abrams2010large,abrams2008efficiency}. TAMD is
based on extensions of~\eqref{eq:overdamped} or~\eqref{eq:inertial} in
which the control parameters $z$ themselves are made dynamical. For
the overdamped Langevin dynamics, these
extended equations read
\begin{equation}
  \label{eq:overdamped_TAMD}
  \left\{ \begin{aligned}
    dq_t & = -\delta^{-1} \nabla_q \Uk(q_t,z_t)\, dt 
    + \sqrt{2(\beta\delta)^{-1}} \, dW_t^q, \\
    dz_t & = -\nabla_z \Uk(q_t,z_t)\, dt + \sqrt{2\bar\beta^{-1}} \, dW_t^z,
  \end{aligned} \right.
\end{equation}
where $W_t^q$ and $W_t^z$ are independent standard $d$- and
$n$-dimensional Wiener processes respectively; while, for the inertial
Langevin dynamics, they read
\begin{equation}
  \label{eq:inertial_TAMD}
  \left\{ \begin{aligned}
    dq_t & = \delta^{-1}  M^{-1} p_t\,dt, \\
    dp_t & = -\delta^{-1}  \nabla_q \Uk(q_t,z_t)\, dt - \delta^{-1} \gamma \, M^{-1}p_t \, dt + \sqrt{2\gamma(\beta\delta)^{-1}} \, dW_t^p, \\
    dz_t & = -\nabla_z \Uk(q_t,z_t)\, dt +
    \sqrt{2\bar\beta^{-1}} \, dW_t^z.
  \end{aligned} \right.
\end{equation}
These equations contain two important parameters: the acceleration
factor, $\delta^{-1}>1$, and the artificial temperature,
$\bar \beta^{-1}$. Taking the first parameter sufficiently large,
i.e. choosing $\delta \ll 1$, allows to accelerate the evolution of $q_t$ or
$(q_t,p_t)$ compared to that of~$z_t$. As a result, one expects that,
given the curent value of $z_t$, $q_t$ or $(q_t,p_t)$ will rapidly
equilibrate according to some invariant measure parametrized by the value~$z_t$, 
and the slow evolution of $z_t$ will only
be affected by the average effect of $q_t$ or $(q_t,p_t)$. More precisely, 
one expects that in the limit as $\delta \to 0$, the evolution of $z_t$
will be effectively captured by that of $\bar z_t$ satisfying
\begin{equation}
  \label{eq:5}
  d\bar z_t = -\nabla_z A(\bar z_t)\, dt +
    \sqrt{2\bar\beta^{-1}} \, dW_t^z,
\end{equation}
where we used~\eqref{eq:2}. A finite time convergence result can be made rigorous using averaging techniques~\cite{PS08}, as already formally done in~\cite{MVE06}. In contrast, we focus here on longtime properties.

In view of~\eqref{eq:5}, and if we take the artificial temperature higher than the physical one, i.e. $\bar \beta < \beta$, \eqref{eq:overdamped_TAMD} and~\eqref{eq:inertial_TAMD} effectively allow one to sample the free energy calculated at the physical temperature using an artificially high temperature to accelerate this sampling. Indeed, the invariant probability measure of~\eqref{eq:5} is
\begin{equation}
  \label{eq:6}
  \onuref(dz)  = \rme^{-\bar \beta A(z)} dz
\end{equation}
where we assumed that the unimportant additive constant that can be
added to $A(z)$ is chosen so that this measure is normalized.  For
$\delta \ll1$, one expects the invariant measures
for~\eqref{eq:overdamped_TAMD} and~\eqref{eq:inertial_TAMD} to be
close to the reference measures
\begin{equation}
  \label{eq:ref_meas_over}
  \nu_{\text{ref}}(dq\,dz) = 
  \rme^{-\beta \Uk(q,z)} \rme^{-(\overline{\beta}-\beta) \Ak(z)} \, dq \, dz
\end{equation}
and
\begin{equation}
  \label{eq:ref_meas_under}
  \nu_{\text{ref}}(dq \, dp \,dz) = 
  \rme^{-\beta \Uk(q,z)} \rme^{-(\overline{\beta}-\beta) \Ak(z)} \, 
  \rme^{-\beta p^T M^{-1}p/2} \, dq \, dp \, dz,
\end{equation}
respectively. The marginal in $z$ of these measures
is~$\bar \nu_{\text{ref}}$. Since $\delta$ is small but finite, in
order to make this statement precise, we need to estimate how close
the invariant measures for~\eqref{eq:overdamped_TAMD}
and~\eqref{eq:inertial_TAMD} are from~\eqref{eq:ref_meas_over}
and~\eqref{eq:ref_meas_under}. We would also like to estimate how fast
the solutions to these equations converge to their invariant
measures. These questions are nontrivial since, with
$\bar \beta \not = \beta$, ~\eqref{eq:overdamped_TAMD}
and~\eqref{eq:inertial_TAMD} are non-reversible stochastic dynamics.

\begin{remark}
For practical computations, the timestep used for the numerical integration of TAMD is limited by the fastest time scale. In this setting, it is in fact better to think of the collective variables as being slowed down by the factor $\delta$ rather than the dynamics in $q$ or $(q,p)$ being accelerated. In the original dynamics we then have two times scales: the $\mathrm{O}(1)$ timescale over which the system evolves, and (assuming there is metastabilty associated with some barrier crossing event of height~$E$), the $\mathrm{O}(\rme^{\beta E})$ Arrhenius timescale. In TAMD, we have three timescales: the $\mathrm{O}(1)$ timescale over which the original system evolves (same as above), the $\mathrm{O}(\delta^{-1})$ timescale over which the collective variables evolve, and the $\mathrm{O}(\delta^{-1}\rme^{\overline{\beta} E})$ time for the crossing event at the artificially high temperature $\overline{\beta}^{-1} > \beta^{-1}$. As long as $\delta^{-1}\rme^{\overline{\beta} E} \ll \rme^{\beta E}$, TAMD provides a computational gain. Typically, this is not so difficult to achieve: the loss via $\delta$ is algebraic while the gain from temperature acceleration is exponential.
\end{remark}

\medskip

The results we obtain on the convergence of TAMD are presented in
Sec.~\ref{sec:mainresults}. Basically, we first establish in Sec.~\ref{sec:expconv} the exponential
convergence of the law of the processes~\eqref{eq:overdamped_TAMD} and~\eqref{eq:inertial_TAMD}, with a 
rate close to the convergence rate of the limiting dynamics~\eqref{eq:5}. 
We next give an asymptotic expansion of their
respective invariant measures in powers of the parameter $\delta$ and
establish their closeness to the known reference measures~\eqref{eq:ref_meas_over} or~\eqref{eq:ref_meas_under}
(see Sec.~\ref{sec:expIM}). We finally state in Sec.~\ref{sec:asvar} a central limit theorem
for the computation of ergodic averages, with an asymptotic variance close to the one 
associated with the limiting dynamics~\eqref{eq:5}. The proofs of these results
are given in Secs.~\ref{sec:proof_thm:unif_exp_cv},
\ref{sec:proof_inv_meas}, and~\ref{sec:proof_expansion}. Beside
establishing the validity range of the TAMD equations
in~\eqref{eq:overdamped_TAMD} and~\eqref{eq:inertial_TAMD}, these
results also permit in principle to adjust the parameters in these
equations to optimize their efficiency.

\section{Set-up and main results}
\label{sec:mainresults}

We focus on situations where the system's positions $q$ are in
a compact domain with periodic boundary conditions $\cD_q = (L\mathbb{T})^d$ 
(as is standard in practical applications in computational statistical physics)
and, for the inertial dynamics, the associated momenta are in $\R^d$. 
We also assume that $z \in \cD_z = \mathbb{T}$
(the extension to $\cD_z = \mathbb{T}^n$ with $n>1$ is straightforward
but makes the notation more cumbersome), and denote the extended
phase space for $q$ and $z$, or $q$, $p$ and $z$, by
$\cE = \cE_x \times \cD_z$, with $\cE_x = \cD_q$ or
$\cD_q \times \R^d$, respectively. In the sequel, we denote by $x$ the
physical variable -- namely, $x=q$ in the overdamped case and
$x=(q,p)$ in the inertial case. 
We also assume that the potential energy function $\Uk:\cD_q\times\cD_z\to\R$ is smooth.
Let us emphasize that it is important for our analysis that $\cD_q \times \cD_z$ is bounded 
since we use several times that functions such as $\nabla U$ and $\nabla (\partial_z U)$ are uniformly bounded.

We write the generators of the
dynamics~\eqref{eq:overdamped_TAMD} and~\eqref{eq:inertial_TAMD} as
\begin{equation}
  \label{eq:generator}
  \sL_\delta = \frac1\delta \cL_0 + \cL_1,
\end{equation}
with
\[
\cL_1 = -\partial_z \Uk(q,z)\partial_z + \bar\beta^{-1} \partial^2_z,
\]
and, in the overdamped case~\eqref{eq:overdamped_TAMD},
\[
\cL_0 = -\nabla_q \Uk(q,z)^T \nabla_q + \beta^{-1} \Delta_q, 
\]
while, in the inertial case~\eqref{eq:inertial_TAMD},
\[
\cL_0 = p^T M^{-1} \nabla_q - \nabla_q \Uk(q,z)^T \nabla_p 
- \gamma p^T M^{-1} \nabla_p + \beta^{-1} \Delta_p.
\]
We also denote by $\cA$ the generator of the limiting dynamics~\eqref{eq:5}:
\begin{equation}
  \label{eq:mathcal_A}
  \cA = -\Ak'(z)\partial_z + \bar \beta^{-1} \partial_z^2,
\end{equation}
and assume that $A(z)$ is shifted by an unimportant additive
constant chosen such that
\[
\bar \nu_{\text{ref}}(\cD_z) =\int_{\cD_z} \rme^{-\overline{\beta}\Ak(z)}\, dz = 1.
\]
We consider all operators as defined on $L^2(\nuref)$ otherwise
explicitly mentioned. In particular, the adjoint $\mathcal{T}^*$ of a closed operator
$\mathcal{T}$ is defined by the following property: for any smooth
functions $\varphi,\phi$ (compactly supported in the $p$ variable for
the inertial dynamics),
\[
\int_\cE \left(\mathcal{T}\varphi\right)\phi \, d\nuref 
= \int_\cE \varphi \left(\mathcal{T}^*\phi\right) \, d\nuref.
\]

For $\bar\beta,\beta,\delta > 0$ fixed, it is
easy to prove that the dynamics~\eqref{eq:overdamped_TAMD}
and~\eqref{eq:inertial_TAMD} have a unique invariant measure
$\nu_\delta(dx\,dz)$ with a smooth, positive density with respect to
the Lebesgue measure. This is clear in the overdamped case~\eqref{eq:overdamped_TAMD} 
since the dynamics is set on a compact space and the Brownian motion acts on all variables
(hence the generator is elliptic). For the inertial case~\eqref{eq:inertial_TAMD}, 
the existence and uniqueness of the invariant measure relies on the fact that 
$1+|p|^2$ is a Lyapunov function, together with some minorization condition on the continuous 
dynamics or a controllability argument combined with the hypoelliptic nature of the generator; see
the review in~\cite[Section~2.4]{LS16} and references therein. In any case, 
the smoothness of the invariant measure is a consequence of the hypoellipticity of the generator.

Denoting by $\psi(t,x,z) \, dx \,dz$ the law of the stochastic processes~\eqref{eq:overdamped_TAMD} or~\eqref{eq:inertial_TAMD} at time $t$ (by hypoellipticity, it can indeed be shown that the law has a smooth density with respect to the Lebesgue measure), the Fokker--Planck equation governing the evolution of the law reads
\[
\partial_t \psi = \sL_\delta^\dagger \psi,
\]
where $\sL_\delta^\dagger$ is the adjoint of $\sL_\delta$ on $L^2(\cE)$. Denoting by $f(t)$ the Radon--Nikodym derivative of $\psi(t,x,z) \, dx \,dz$ with respect to $\nu_{\rm ref}(dx\,dz)$, the Fokker--Planck equation can be rewritten on $L^2(\nuref)$ as
\begin{equation}
  \label{eq:FP_L2_nuref}
  \partial_t f(t) = \sL_\delta^* f(t).
\end{equation}
In particular, introducing
\begin{equation}
  \label{eq:def_h_delta}
  h_\delta = \frac{d\nu_\delta}{d\nu_{\text{ref}}},
\end{equation}
the function $h_\delta$ is smooth and satisfies the stationary Fokker--Planck equation
\[
\sL_\delta^* h_\delta = 0.
\]  
We expect that $h_\delta$ can be expanded at any order in~$\delta$, with a leading term which is~1.  This result is established below in Theorem~\ref{thm:inv_meas}. We also expect that $f(t)$ converges to $h_\delta$, as made precise in Theorem~\ref{thm:unif_exp_cv}.

\subsection{Exponential convergence of the law}
\label{sec:expconv}

Our first set of results concerns the convergence of the law of the processes towards the invariant measure. In view of~\eqref{eq:FP_L2_nuref}, $f(t) = \rme^{t \sL_\delta^*}f(0)$, so that the convergence result can be rephrased as a convergence result for the evolution semigroup $\rme^{t \sL_\delta^*}$ on $L^2(\nuref)$. We state the result for general densities $f \in L^2(\nuref)$, which are not necessarily such that $f \geq 0$ and $\int_\cE f \, d\nuref = 1$ as would be the case for probabilities densities. 

\begin{theorem}[Uniform exponential convergence]
\label{thm:unif_exp_cv}
Denote by $R_{\overline{\beta}}>0$ the constant in the Poincar\'e inequality
satisfied by the the marginal measure $\onuref(dz)$ in~\eqref{eq:6}: for all $\varphi \in H^1(\cD_z)$ such that
\[
\int_{\cD_z} \varphi(z) \, \onuref(dz) =0,
\]
it holds 
\begin{equation}
  \label{eq:Poincare_marginal}
  \int_{\cD_z} \left|\varphi(z)\right|^2 \, \onuref(dz) \leq 
  \frac{1}{R_{\overline{\beta}}^2}\int_{\cD_z} |\varphi'(z)|^2 \, \onuref(dz). 
\end{equation}
There is $\delta_*>0$ such that, for all
$\delta \in (0,\delta_*]$, it holds $h_\delta \in
L^2(\nuref)$. Moreover, for any $\varepsilon > 0$, there exist
$\delta_\varepsilon > 0$ sufficiently small and
$c_\varepsilon \in [1,2]$ such that, for any
$\delta \in (0,\delta_\varepsilon]$ and any $f \in L^2(\nuref)$,
\begin{equation}
\label{eq:exp_contraction_to_h_delta}
\begin{aligned}
  \forall t \geq 0, \qquad &\left\| \rme^{t\sL_\delta^*} f -
    \left({\textstyle\int_\cE} f \, d\nuref\right) h_\delta
  \right\|_{L^2(\nuref)}\\
  & \qquad \leq c_\varepsilon
  \rme^{-(\lambda-\varepsilon) t} \left\|f - \left({\textstyle\int_\cE} f \,
      d\nuref\right)h_\delta \right\|_{L^2(\nuref)},
\end{aligned}
\end{equation}
with 
\begin{equation}
  \label{eq:def_lambda}
  \lambda = \frac{R_{\overline{\beta}}^2}{\overline{\beta}}.
\end{equation}
For the overdamped dynamics~\eqref{eq:overdamped_TAMD}, the constant
$c_\varepsilon$ can be chosen equal to~1.
\end{theorem}

This result is proven in Sec.~\ref{sec:proof_thm:unif_exp_cv}. It shows that the convergence rate of the law of the TAMD dynamics can be made arbitrarily close to the convergence rate of the law of the limiting dynamics~\eqref{eq:5} upon choosing $\delta$ sufficiently small. Recall indeed that the Poincar\'e inequality~\eqref{eq:Poincare_marginal} implies that, for a function $g : \cD_z \to \R$ belonging to $L^2(\onuref)$, the following convergence result holds for the semigroup $\rme^{t \cA^*}$ associated with~\eqref{eq:5}: 
\[
\left\| \rme^{t \cA^*}g - \bone \right\|_{L^2(\onuref)} \leq \rme^{-\lambda t} \left\| g - \bone \right\|_{L^2(\onuref)},
\]
see for instance~\cite{bakry-gentil-ledoux-14}.

From a technical viewpoint, the proof uses techniques similar to ones used to prove the convergence of the semigroups associated with overdamped or inertial Langevin dynamics at equilibrium, namely Gronwall-type estimates for the squared norm in $L^2(\nuref)$ of solutions of the Fokker--Planck equation~\eqref{eq:FP_L2_nuref}; see for instance~\cite{bakry-gentil-ledoux-14} in the overdamped case, and~\cite{HP08,Villani09,DMS09,DMS15} in the inertial case; as well as the review provided in~\cite[Section~2]{LS16}. 

The contraction estimates~\eqref{eq:exp_contraction_to_h_delta} also imply resolvent estimates on the generator and its adjoint. To state these estimates, we introduce the projectors
\[
P_\delta f = f - {\textstyle\int_\cE} f \, d\nu_\delta,
\]
and
\[
Q_\delta f = f - \left({\textstyle\int_\cE}  f \, d\nuref\right)h_\delta,
\]
as well as the images of $L^2(\nuref)$ by these projectors: 
\[
\begin{aligned}
P_\delta L^2(\nuref) & = \left\{ f \in L^2(\nuref) \, \left| \, P_\delta f = f \right.\right\}, \\
Q_\delta L^2(\nuref) & = \left\{ f \in L^2(\nuref) \, \left| \, Q_\delta f = f \right.\right\}.
\end{aligned}
\]
Finally, we denote by $\mathcal{B}(E)$ the Banach space of bounded,
linear operators on a Banach space~$E$, and by
$\|\cdot\|_{\mathcal{B}(E)}$ the associated norm:
\[
\|\mathcal{A}\|_{\mathcal{B}(E)} = 
\sup_{\varphi \in E\backslash\{0\}} \frac{\|\mathcal{A}\varphi\|_E}{\|\varphi\|_E}.
\]
Then Theorem~\ref{thm:unif_exp_cv} implies:
 
\begin{corollary}
\label{corr:resolvent_estimates}
For any $\varepsilon \in (0,\lambda)$, there exists
$\delta_\varepsilon > 0$ such that, for any
$\delta \in (0,\delta_\varepsilon]$,
\[
\begin{aligned}
  \left\| \left(\sL_\delta^*\right)^{-1} \right\|_{\cB(Q_\delta
    L^2(\nuref))} 
  & \leq \frac{c_\varepsilon}{\lambda-\varepsilon}, \\
  \left\| \sL_\delta^{-1} \right\|_{\cB(P_\delta L^2(\nuref))} 
  & \leq \frac{c_\varepsilon}{\lambda-\varepsilon},
\end{aligned}
\]
where $c_\varepsilon$ is the constant appearing in Theorem~\ref{thm:unif_exp_cv}.
\end{corollary}

This result is proved in
Section~\ref{sec:proof_corollary:resolvent_estimates}. An implicit
statement in the above corollary (made explicit in the proof) is that
the operator $\sL_\delta^*$ commutes with $Q_\delta$ and therefore
stabilizes $Q_\delta L^2(\nuref)$; and similarly that $\sL_\delta$
commutes with $P_\delta$ and therefore stabilizes
$P_\delta L^2(\nuref)$.

\subsection{Perturbative expansion of the invariant measure}
\label{sec:expIM}

A second set of results proven in this paper concerns the expansion in
powers of~$\delta$ of the invariant measure
$d\nu_{\delta} = h_\delta \, d\nuref$, relying on asymptotic expansions
and the resolvent bounds from
Corollary~\ref{corr:resolvent_estimates}. In order to state it,
we introduce the projection operator with respect to the conditional measure
in~$x$ at fixed~$z$:
\[
(\Pi_z \varphi)(z) = \int_{\cE_x} \varphi(x,z) \, \nurefz(dx).
\]
where
\[
\nurefz(\cdot) = \frac{\nuref(\cdot\  dz)}{\onuref(dz)}, 
\]
with $\onuref$ defined in~\eqref{eq:6}. Note that $\Pi_z$ is an orthogonal projector on~$L^2(\nuref)$, i.e. $\Pi_z^2 = \Pi_z = \Pi_z^*$.

\begin{theorem}
  \label{thm:inv_meas}
  There exist $C,\delta_*>0$ such that
  \[
  \forall \delta \in (0,\delta_*], 
    \qquad 
    \left\|h_\delta - (1 + \delta \fh ) \right\|_{L^2(\nuref)} \leq C \delta^2,
    \]
    where
    \[
    \fh = \left(\Pi_z \cA^{-1}\Pi_z \cL_1^* - 1\right)(\cL_0^*)^{-1} 
    (1-\Pi_z)\cL_1^* \bone \in L^2(\nuref), 
    \qquad
    \int_\cE \fh \, d\nuref = 0.
    \]
\end{theorem}

It is in fact possible to construct higher order correction terms in
order to have remainders of arbitrary power of~$\delta$.  The proof of
Theorem~\ref{thm:inv_meas} is given in Sec.~\ref{sec:proof_inv_meas}, where it shown in particular that $\fh$ makes sense.

\subsection{Characterization of the asymptotic variance}
\label{sec:asvar}

In practice, average properties are estimated using ergodic means. For
a given observable~$\varphi \in L^1(\nu_\delta)$,
\[
\widehat{\varphi}_T = \frac1T \int_0^T \varphi(x_t,z_t) \, dt
\]
almost surely converges to $\mathbb{E}_{\nu_\delta}(\varphi)$ as
$T \to \infty$. Moreover, for $\varphi \in L^2(\nuref)$, the fact that
the Poisson equation
\begin{equation}
  \label{eq:Poisson_variance_prop}
  -\sL_\delta \Phi_\delta = P_\delta \varphi, \qquad \int_\cE \Phi_\delta \, d\nuref = 0,
\end{equation}
admits a unique solution in $L^2(\nuref)$ by
Corollary~\ref{corr:resolvent_estimates} ensures that a central limit
theorem holds for any initial distribution
(see~\cite{Bhattacharya82}), with associated asymptotic variance
\[
\sigma_{\varphi,\delta}^2 = 
\lim_{T \to \infty} T \var_{\nu_\delta}
\left(\widehat{\varphi}_T\right) 
= 2 \int_\cE (P_\delta \varphi) \left(-\sL_\delta^{-1}P_\delta
  \varphi\right) \, d\nu_\delta 
= 2 \int_\cE (P_\delta \varphi) \Phi_\delta h_\delta \, d\nuref.
\]

The following result shows that, with some additional
regularity/integrability properties on the observable, the asymptotic
variance $\sigma_{\varphi,\delta}^2$ is close to the asymptotic variance $\sigma_{\varphi,{\rm ref}}^2$ of the
reference dynamics, where $z_t$ evolves according to the dynamics~\eqref{eq:5} with generator $\cA$
while $x_t$ is instantaneously equilibrated according to the corresponding conditional measure~$\nurefz(dx)$. 
More precisely,
\[
\sigma_{\varphi,{\rm ref}}^2 = 2 \int_\cE \left(\Pi_z P_0 \varphi\right)
\left(-\cA^{-1}\Pi_z P_0 \varphi \right) d\onuref,
\qquad 
P_0 \varphi = \varphi - {\textstyle\int_\cE}  \varphi \, d\nuref. 
\]
A quantitative estimate of the error however requires considering more
regular observables, namely in $H^2(\nuref)$ rather than
$L^2(\nuref)$. We state this result as:

\begin{theorem}
  \label{thm:expansion_sL_delta_inv}
  There exist $C,\delta_*>0$ such that, for any $\varphi \in H^2(\nuref)$,
  \[
  \forall \delta \in (0,\delta_*], \qquad 
  \left| \sigma_{\varphi,\delta}^2 - \sigma_{\varphi,{\rm ref}}^2 \right| \leq C \delta \|\varphi\|_{L^2(\nuref)} \|\varphi\|_{H^2(\nuref)}.
  \]
\end{theorem}

A careful inspection of the proof would allow to make precise the
leading order term in the difference of the asymptotic variances. We
refrain here from doing so since the expression of this leading order
correction is somewhat cumbersome. 
Theorem~\ref{thm:expansion_sL_delta_inv} directly follows from two points: (i)
the expansion for the invariant measure provided by
Theorem~\ref{thm:inv_meas}, which shows in particular that there exists $K>0$ such
that, for $\delta>0$ sufficiently small,
\begin{equation}
  \label{eq:diff_Pdelta_P0}
  \left| P_\delta \varphi - P_0 \varphi \right| \leq K \delta \|\varphi\|_{L^2(\nuref)};
\end{equation}
and (ii) an expansion for the unique solution
$\Phi_\delta \in L^2(\nuref)$ of the Poisson equation~\eqref{eq:Poisson_variance_prop}
as given by the following result:

\begin{proposition}
  \label{prop:approx_solution_Poisson}
  There exists $C,\delta_*>0$ such that, for any
  $\varphi \in H^2(\nuref)$, the unique solution
  $\Phi_\delta \in L^2(\nuref)$ of the Poisson equation~\eqref{eq:Poisson_variance_prop} satisfies 
  \[
  \forall \delta \in (0,\delta_*], \qquad 
  \left\|\Phi_\delta + \mathcal{A}^{-1} \Pi_z P_0 \varphi\right\|_{L^2(\nuref)}  
  \leq C \delta \|\varphi\|_{H^2(\nuref)}.
  \]
\end{proposition}
The proof of this proposition is given in Sec.~\ref{sec:proof_expansion}.

\section{Proof of Theorem~\ref{thm:unif_exp_cv}}
\label{sec:proof_thm:unif_exp_cv}

We start by noting that, for the overdamped
dynamics~\eqref{eq:overdamped_TAMD}, the operator $\sL_\delta$ can be
rewritten as
\[
\sL_\delta = -\frac{1}{\beta \delta} \nabla_q^* \nabla_q 
-\frac{1}{\overline{\beta}} \partial_z^* \partial_z 
+ \left(\frac{\beta}{\overline{\beta}}-1\right) (\partial_z \Uk - \Ak')\partial_z,
\]
since
\[
\int_\cD (\partial_z f) g \, d\nuref = - \int_\cD f (\partial_z g) \,d\nuref 
+ \int_\cD f g \, \left[\beta \partial_z \Uk 
+ (\overline{\beta}-\beta)\Ak'\right] d\nuref,
\]
so that
$\partial_z^* = -\partial_z + \beta \partial_z \Uk +
(\overline{\beta}-\beta)\Ak'$. Similarly, for the inertial
dynamics~\eqref{eq:inertial_TAMD},
\[
\sL_\delta = \frac{1}{\delta}\left(\mathcal{L}_{\rm ham} 
  - \frac{\gamma}{\beta} \nabla_p^* \nabla_p \right) 
-\frac{1}{\overline{\beta}} \partial_z^* \partial_z 
+ \left(\frac{\beta}{\overline{\beta}}-1\right) (\partial_z \Uk - \Ak')\partial_z.
\]
where
\[\mathcal{L}_{\rm ham} = p^T M^{-1} \nabla_q - \nabla_q \Uk(q,z)^T \nabla_p.
\]
The first two terms in the expressions for $\sL_\delta$ will provide a
control on the derivatives with respect to~$x$ and~$z$, hence on the norm 
of functions with
average~0 with respect to~$\nuref$. The last term can be considered as a perturbation. 
In all this proof, we consider functions to be real valued for the 
simplicity of notation. It is straightforward to extend the proofs to complex valued functions 
upon taking real values of the scalar products when necessary.

In order to encompass both overdamped and inertial dynamics,
we introduce the following scalar product, suggested in~\cite{DMS09,DMS15}
(see also~\cite{IOS17,RS17} for rewritings in a $L^2(\nuref)$ setup):
\[
\left( \varphi,\phi \right)_\eta = \int_{\cD_z} \llll
\varphi(\cdot,z), 
\phi(\cdot,z) \rrrr_\eta \, \onuref(dz) ,
\]
where, for some parameter $\eta \in (0,1)$ to be made precise, 
\[
\begin{aligned}
\llll \varphi(\cdot,z), \phi(\cdot,z) \rrrr_\eta 
& = \left \langle \varphi(\cdot,z), \phi(\cdot,z) \right\rangle_{L^2(\nurefz)} \\
& \quad - \eta \left \langle (B\varphi)(\cdot,z), \phi(\cdot,z)
\right\rangle_{L^2(\nurefz)} \\
& \quad - \eta \left \langle \varphi(\cdot,z), (B\phi)(\cdot,z) \right\rangle_{L^2(\nurefz)}.
\end{aligned}
\]
In this expression, $B$ is the following bounded operator acting on variables~$q,p$ only:
\[
B = \left(1 + \frac{1}{\beta m} \nabla_q^* \nabla_q \right)^{-1} \Pi_p
\cL_{\rm ham}, 
\]
where $M \leq m \mathrm{Id}$ (in the sense of symmetric matrices) and
\[
(\Pi_p\varphi)(q) = \left(\frac{\beta}{2\pi}\right)^{d/2} (\det M)^{-1/2} 
  \int_{\R^d} \varphi(q,p)\, \rme^{-\frac12\beta p^T M^{-1}p} \, dp.
\]
The operator $B$ only makes sense for the inertial dynamics, so that we
always set $\eta = 0$ for the overdamped  dynamics. 
The modified scalar product is equivalent to the standard one since
\[
\|B\|_{\mathcal{B}(L^2(\nuref))} \leq \frac12,
\]
see~\cite{DMS09,DMS15,IOS17,RS17}; more precisely
\begin{equation}
  \label{eq:scalar_product}
  (1-\eta) \| \varphi \|^2_{L^2(\nuref)} \leq 
  \left( \varphi, \varphi \right)_\eta \leq (1+\eta) \| \varphi \|^2_{L^2(\nuref)}.
\end{equation}
The interest of the modified scalar product lies in the following coercivity estimates for the virtual dynamics (i.e. the dynamics on $q$ or $(q,p)$ with $z$ fixed).

\begin{lemma}
  \label{lem:conditioned_coercivity}
  For any $\eta > 0$, there exists $\rho > 0$ such that, for all smooth and compactly supported function~$f$,
  \[
  \llll (1-\Pi_z)f(\cdot,z), (1-\Pi_z)f(\cdot,z) \rrrr_\eta \leq 
  \frac{1}{\rho} \llll -\cL_0^* f(\cdot,z), f(\cdot,z) \rrrr_\eta.
  \]
  In particular,
  \begin{equation}
    \label{eq:Poincare_in_x}
    \| f - \Pi_z f \|_{L^2(\nuref)}^2 \leq \frac{1}{(1-\eta)\rho} \left( -\cL_0^* f, f\right)_\eta.
  \end{equation}
\end{lemma}

The second statement follows directly from the first one by
integrating in~$z$ and using~\eqref{eq:scalar_product}. The proof of
the first statement for the overdamped dynamics is a direct
consequence of the following Poincar\'e inequality, which is uniform in the
parameter~$z$ (relying on the smoothness of~$U$ and the fact that $\cD_q,\cD_z$ are both bounded): 
there exists $R_q>0$ such that, for any $f \in H^1(\nuref)$,
\begin{equation}
  \label{eq:Poincare_in_q}
  \int_{\cD_q} \left|f(q,z) - (\Pi_z f)(z)\right|^2 \, \nurefz(dq) \leq \frac{1}{R^2_q}\int_{\cD_q} \left|\nabla_q f(q,z) \right|^2 \, \nurefz(dq).
\end{equation}
In this case, it is possible to consider $\eta = 0$ and $\rho = R_q^2/\beta$. For the inertial dynamics, the first statement is proved as in~\cite{DMS09,DMS15}, relying also on the above uniform in~$z$ Poincar\'e inequality. A careful inspection of the proof for the inertial dynamics shows that $\rho$ is of order~$\eta$ (see for instance~\cite{RS17}).

\begin{remark}
  In both cases, the parameter $\rho>0$ quantifies the mixing rate of
  the dynamics, uniformly in~$z$: the larger it is, the faster the
  mixing is. The parameter $\rho$ would typically be small if there is some metastability in the dynamics
  for a given value of~$z$.
\end{remark}

The estimates of Lemma~\ref{lem:conditioned_coercivity} show how to
bound $\|(1-\Pi_z)f\|_{L^2(\nuref)}$ using only the part $\cL_0$ of
the dynamics depending on the spatial variables $x = q$ or $(q,p)$.
Now, recall that
\begin{equation}
  \label{eq:decomposition_f_Pi_z_1_Pi_z}
  \|f\|^2_{L^2(\nuref)} = \|\Pi_z f\|^2_{L^2(\nuref)} + \|(1-\Pi_z) f\|^2_{L^2(\nuref)}.
\end{equation}
In order to control the full norm $\|f\|_{L^2(\nuref)}$ for functions in
\[
L^2_0(\nuref) = P_0 L^2(\nuref) = \left\{ f \in L^2(\nuref) \,
\left| {\textstyle\int_\cE} f \, \nuref = 0 \right. \right\},
\]
we therefore need to control the missing part $\Pi_z f$, using the part of the
generator depending on derivatives in~$z$. Before we state such a
result in Lemma~\ref{lem:marginal_coercivity} below, we mention 
commutator identities which are useful for inertial Langevin dynamics. The commutator of two operators $A,B$ is defined as $[A,B] = AB-BA$. 

\begin{lemma}
  \label{lem:commutation_B_partial_z}
  The commutators $[B,\partial_z]$ and $[B^*,\partial_z]$ are bounded operators: there exists $K_{\rm c} \in R_+$ such that
  \[
  \|[B,\partial_z] \|_{\mathcal{B}(L^2(\nuref))} \leq 
  K_{\rm c}, \qquad \| [B^*,\partial_z]\|_{\mathcal{B}(L^2(\nuref))} \leq K_{\rm c}. 
  \]
  Moreover, 
  \begin{equation}
  \label{eq:commutation_Pi_z}
  \Pi_z \partial_z f - \partial_z \Pi_z f = \beta \Pi_z\left[(\partial_z U - \Ak')f\right]. 
\end{equation}
\end{lemma}

The proof of the first inequality is postponed to Section~\ref{sec:commutator}; the
second one is easily checked by a direct computation.
 
\begin{lemma}
  \label{lem:marginal_coercivity}
  For any $\alpha > 0$, there is $\eta>0$ and $C_\alpha > 0$ such that, for all smooth and compactly supported function~$f \in L^2_0(\nuref)$,
  \[
    \|\Pi_z f\|_{L^2(\nuref)}^2 \leq
    \frac{1}{R_{\overline{\beta}}^2-\alpha}
    \left( \partial_z^* \partial_z f, f \right)_\eta + C_\alpha
      \left\|(1-\Pi_z)f\right\|_{L^2(\nuref)}^2.
  \]
  For overdamped dynamics, it is possible to consider $\alpha = 0$ and $\eta = 0$.
\end{lemma}

\begin{proof}
  Using first the Poincar\'e inequality~\eqref{eq:Poincare_marginal}
  (since $f$ is of zero average with respect to $\nuref$ hence $\Pi_z f$
  has zero average with respect to $\onuref$),
  and then~\eqref{eq:commutation_Pi_z},
\[
\begin{aligned}
\|\Pi_z f\|_{L^2(\nuref)} & 
\leq \frac{1}{R_{\overline{\beta}}} \|\partial_z \Pi_z f\|_{L^2(\nuref)} \\
& \leq \frac{1}{R_{\overline{\beta}}} \left[ \|\Pi_z \partial_z
  f\|_{L^2(\nuref)} 
+ \beta \left\|\Pi_z\left((\partial_z \Uk-\Ak') f\right)\right\|_{L^2(\nuref)}\right] \\
& \leq \frac{1}{R_{\overline{\beta}}} \left[ \| \partial_z
  f\|_{L^2(\nuref)} 
+ \beta \left\|\Pi_z(\partial_z \Uk-\Ak')(1-\Pi_z)f\right\|_{L^2(\nuref)} \right],
\end{aligned}
\]
where we used the fact that $\Pi_z (\partial_z \Uk - \Ak')\Pi_z = 0$
since $\partial_z \Uk - \Ak'$ has zero average with respect to
$\rme^{-\beta (\Uk(q,z)-\Ak(z))} \, dq$. Therefore,
\begin{equation}
\label{eq:estimate_Pi_z_partial_z}
\|\Pi_z f\|_{L^2(\nuref)} \leq \frac{1}{R_{\overline{\beta}}} \left[ \| \partial_z
  f\|_{L^2(\nuref)} + \beta \left\|\partial_z \Uk-\Ak'\|_{L^\infty} \|(1-\Pi_z)f\right\|_{L^2(\nuref)} \right],
\end{equation}
which already gives the result for overdamped dynamics.

For inertial Langevin, we use~\eqref{eq:scalar_product} to write
\begin{equation}
  \label{eq:bound_Pi_z_f}
  \|\Pi_z f\|_{L^2(\nuref)}\leq \frac{1}{R_{\overline{\beta}}} 
  \left[ \frac{\dps\left( \partial_z f, \partial_z f
      \right)_\eta^{1/2}}{1-\eta} 
    + \beta \|\partial_z \Uk-\Ak'\|_{L^\infty} \left\|(1-\Pi_z)f\right\|_{L^2(\nuref)} \right].
\end{equation}
We next note that 
\[
\Big( \partial_z f, \partial_z f \Big)_\eta = \Big( \partial_z^* \partial_z f, f \Big)_\eta - \eta \left\langle [B,\partial_z]f, \partial_z f \right\rangle_{L^2(\nuref)} - \eta \left\langle \partial_z f, [B^*,\partial_z]f \right\rangle_{L^2(\nuref)}.
\]
In view of Lemma~\ref{lem:commutation_B_partial_z},
\[
\begin{aligned}
  & \Big( \partial_z f, \partial_z f \Big)_\eta\\
  & \leq \Big( \partial_z^* \partial_z f, f \Big)_\eta 
  + 2 \eta K_{\rm c} \| \partial_z f\|_{L^2(\nuref)} \| f\|_{L^2(\nuref)} \\
  & \leq \Big( \partial_z^* \partial_z f, f \Big)_\eta 
  + \eta K_{\rm c} \left[ \| \partial_z f\|^2_{L^2(\nuref)} 
    + \| \Pi_z f\|_{L^2(\nuref)}^2 + \| (1-\Pi_z) f\|_{L^2(\nuref)}^2 \right] \\
  & \leq \Big( \partial_z^* \partial_z f,f \Big)_\eta 
  + \eta(1+\eta)K_{\rm c}\Big( \partial_z f, \partial_z f \Big)_\eta \\
  & \qquad + \eta K_{\rm c} \left[ \| \Pi_z f\|_{L^2(\nuref)}^2 + \| (1-\Pi_z) f\|_{L^2(\nuref)}^2 \right],
\end{aligned}
\]
so that, for $\eta > 0$ sufficiently small,
\[
\begin{aligned}
& \Big( \partial_z f, \partial_z f \Big)_\eta \\
& \leq \frac{1}{1-\eta(1+\eta)K_{\rm c}}\left[\Big( \partial_z^* \partial_z f,f \Big)_\eta + \eta K_{\rm c} \left(\| \Pi_z f\|_{L^2(\nuref)}^2 + \| (1-\Pi_z) f\|_{L^2(\nuref)}^2\right) \right].
\end{aligned}
\]
The final result then immediately follows from this inequality
and~\eqref{eq:bound_Pi_z_f}, upon taking $\eta>0$ sufficiently small.
\end{proof}

We finally give estimates on the non-symmetric part of the adjoint of
\[
\cL_1 = -\frac{1}{\overline{\beta}} \partial_z^* \partial_z + 
\left(\frac{\beta}{\overline{\beta}}-1\right) (\partial_z \Uk - \Ak')\partial_z,
\]
which is proportional to the adjoint of $\wcL = (\partial_z \Uk - \Ak')\partial_z$.

\begin{lemma}
  \label{lem:extra_term_to_control}
  For any $\alpha > 0$, there is $\eta > 0$ and $C_\alpha \in \R_+$
  such that, for all smooth and compactly supported functions~$f$, 
  \[
    \left| \left( \wcL^* f, f\right)_\eta \right| \leq \alpha
    \left(\left\| f \right\|_{L^2(\nuref)}^2 + \left\| \partial_z f
      \right\|_{L^2(\nuref)}^2\right) + C_\alpha
    \|(1-\Pi_z)f\|_{L^2(\nuref)}^2.
  \]
\end{lemma}

\begin{proof}
  We start by noting that 
  \[
  \begin{aligned}
    \left \langle \wcL^* f, f \right\rangle_{L^2(\nuref)} & 
    = \left \langle \wcL f, f \right\rangle_{L^2(\nuref)} \\
    & = \left\langle (\partial_z \Uk - \Ak')\partial_z f, 
      (1-\Pi_z)f\right\rangle_{L^2(\nuref)} \\
    & \quad
    + \left\langle \partial_z f, (\partial_z \Uk - \Ak')\Pi_zf\right\rangle_{L^2(\nuref)},
  \end{aligned}
  \]
  the first term being controlled by
  $\|\partial_z f\|_{L^2(\nuref)} \|\partial_z \Uk - \Ak'\|_{L^\infty}
  \|(1-\Pi_z)f\|_{L^2(\nuref)}$. For the second one, we use the commutation
  rule~\eqref{eq:commutation_Pi_z} and the orthogonality of $\Pi_z$ and $1-\Pi_z$ in $L^2(\nuref)$ 
  to write
  \[
  \begin{aligned}
    \left\langle \partial_z f, (1-\Pi_z) h
    \right\rangle_{L^2(\nuref)} & = \left\langle (1-\Pi_z)\partial_z
    f, (1-\Pi_z) h \right\rangle_{L^2(\nuref)} \\
    & = \left\langle \partial_z (1-\Pi_z) f, (1-\Pi_z) h
    \right\rangle_{L^2(\nuref)}.
    \end{aligned}
  \]
  Applying this result to $h = (\partial_z \Uk - \Ak')\Pi_z f$ (which is such that $(1-\Pi_z)h = h$
  since $\Pi_z(\partial_z \Uk - \Ak') = 0$),
  \[
  \begin{aligned}
    \left\langle \partial_z f, (\partial_z \Uk - \Ak')\Pi_zf\right\rangle_{L^2(\nuref)} & = \left\langle \partial_z (1-\Pi_z) f, (\partial_z \Uk - \Ak')\Pi_zf\right\rangle_{L^2(\nuref)} \\
    & = \left\langle (1-\Pi_z) f, \partial_z^* \left[(\partial_z \Uk - \Ak')\Pi_zf \right]\right\rangle_{L^2(\nuref)},
  \end{aligned}
  \]
so that 
\begin{equation}
  \label{eq:estimate_wcL*}
  \left|\left\langle \partial_z f, (\partial_z \Uk - \Ak')
      \Pi_zf\right\rangle_{L^2(\nuref)}\right| \leq 
  \left\| (1-\Pi_z) f \right\|_{L^2(\nuref)} \left\|\wcL^* \Pi_zf \right\|_{L^2(\nuref)}. 
\end{equation}
Noting that
\begin{equation}
\label{eq:def_W0}
\wcL^* g = -\wcL g + W, 
\end{equation}
where
\begin{equation}
  \label{eq:def_W} 
  W = -(\partial^2_z \Uk - \Ak'') + (\beta \partial_z \Uk 
  + (\overline{\beta}-\beta)\Ak')(\partial_z \Uk - \Ak'),
\end{equation}
the last factor on the right-hand side of~\eqref{eq:estimate_wcL*} can be bounded as  
(using again~\eqref{eq:commutation_Pi_z})
\[
\begin{aligned}
  \left\| \wcL^* \Pi_zf \right\|_{L^2(\nuref)} 
  & \leq \left\| (\partial_z \Uk - \Ak') \partial_z \Pi_zf
  \right\|_{L^2(\nuref)} 
  + \left\| W \right\|_{L^\infty} \|\Pi_zf\|_{L^2(\nuref)} \\
  & \leq \left\| \partial_z \Uk - \Ak' \right\|_{L^\infty} 
  \left\| \partial_z \Pi_z f \right\|_{L^2(\nuref)} 
  +  \left\| W \right\|_{L^\infty} \|\Pi_zf\|_{L^2(\nuref)} \\
  & \leq \left\| \partial_z \Uk - \Ak' \right\|_{L^\infty} 
  \left[\left\| \partial_z f \right\|_{L^2(\nuref)} 
    + \beta \left\| \partial_z \Uk - \Ak' \right\|_{L^\infty} 
    \left\| f \right\|_{L^2(\nuref)}\right]\\
  & \quad +  \left\| W \right\|_{L^\infty} \|\Pi_zf\|_{L^2(\nuref)}.
\end{aligned}
\]

The two extra terms in the scalar product $(\cdot,\cdot)_\eta$ for the inertial dynamics
can be estimated in a crude manner as
\[
\begin{aligned}
& \left|\left \langle B \wcL^* f, f
  \right\rangle_{L^2(\nuref)}\right|\\
& \qquad \leq \frac12 \left\|\wcL^* f\right\|_{L^2(\nuref)}\left\|f \right\|_{L^2(\nuref)} \\
& \qquad \leq \frac12 \left( \|\partial_z \Uk - \Ak'\|_{L^\infty} \left\|\partial_z f \right\|_{L^2(\nuref)} + \left\| W \right\|_{L^\infty} \left\|f \right\|_{L^2(\nuref)}\right)\left\|f \right\|_{L^2(\nuref)},
\end{aligned}
\]
and the same upper bound for the other term
$|\langle \wcL^* f, Bf \rangle_{L^2(\nuref)}\|$. In conclusion,
\[
\begin{aligned}
  \left| \Big( \wcL^* f, f\Big)_\eta \right| & 
  \leq \|\partial_z f\|_{L^2(\nuref)} \|\partial_z \Uk -
  \Ak'\|_{L^\infty} 
  \|(1-\Pi_z)f\|_{L^2(\nuref)} \\
  & \quad + \left\| \partial_z \Uk - \Ak' \right\|_{L^\infty} 
  \left\| \partial_z f \right\|_{L^2(\nuref)} \|(1-\Pi_z)f\|_{L^2(\nuref)} \\
  & \quad + \left[ \beta \left\| \partial_z \Uk - \Ak' \right\|^2_{L^\infty} 
    + \|W\|_{L^\infty} \right] \left\| f \right\|_{L^2(\nuref)}\|(1-\Pi_z)f\|_{L^2(\nuref)}\\
  & \quad + \eta \|\partial_z \Uk - \Ak'\|_{L^\infty} 
  \left\|\partial_z f \right\|_{L^2(\nuref)} \left\|f \right\|_{L^2(\nuref)} 
  + \eta \left\| W \right\|_{L^\infty} \left\|f \right\|_{L^2(\nuref)}^2.
\end{aligned}
\]
The final estimate follows by Young's inequality, upon taking $\eta>0$
sufficiently small.
\end{proof}

As a consequence of the previous results, we obtain the following key
coercivity property for the complete dynamics.

\begin{corollary}
  \label{corr:global_coercivity}
  For any $\varepsilon > 0$, there is $\eta>0$ and
  $\delta_\varepsilon > 0$ such that, for all smooth and compactly
  supported function~$f \in L^2_0(\nuref)$,
  \[
  \forall \delta \in (0,\delta_\varepsilon), \qquad 
  \left( -\sL_\delta^* f,f \right)_\eta \geq (\lambda-\varepsilon) \left(f,f \right)_\eta,
  \]
  where $\lambda$ is defined in~\eqref{eq:def_lambda}.
\end{corollary}

\begin{proof}
The proof of this result is based on Lemmas~\ref{lem:conditioned_coercivity}, \ref{lem:marginal_coercivity} and~\ref{lem:extra_term_to_control}. The value of $\eta>0$ is first fixed by Lemmas~\ref{lem:marginal_coercivity} and~\ref{lem:extra_term_to_control}. More precisely, for any $\alpha > 0$, there exist $\eta > 0$ and a constant $C_\alpha > 0$ such that
\[
\begin{aligned}
\left(-\cL_1^* f,f\right)_\eta & \geq (\lambda - \alpha) \|\Pi_z f\|_{L^2(\nuref)}^2 - \alpha\left(\|\partial_z f\|_{L^2(\nuref)}^2 + \|f\|_{L^2(\nuref)}^2\right) \\ 
& \quad - C_\alpha \|(1-\Pi_z) f\|_{L^2(\nuref)}^2.
\end{aligned}
\]
We next use $\|f\|_{L^2(\nuref)}^2 \leq (1+\eta)(f,f)_\eta$ as well as the bound~\eqref{eq:estimate_Pi_z_partial_z} and~\eqref{eq:decomposition_f_Pi_z_1_Pi_z}. This shows that, for any $\varepsilon>0$, there exist $\eta>0$ and $C>0$ such that 
\[
\left(-\cL_1^* f,f\right)_\eta \geq \left(\lambda - \frac{\varepsilon}{2}\right) \|\Pi_z f\|_{L^2(\nuref)}^2 - C \|(1-\Pi_z) f\|_{L^2(\nuref)}^2.
\]
Choosing $\delta > 0$ sufficiently small in Lemma~\ref{lem:conditioned_coercivity}, 
\[
\begin{aligned}
\left(-\sL_\delta^* f,f\right)_\eta & \geq \left(\lambda - \frac{\varepsilon}{2}\right) \left( \|\Pi_z f\|_{L^2(\nuref)}^2 + \|(1-\Pi_z) f\|_{L^2(\nuref)}^2\right) \\
& \geq (1-\eta)\left(\lambda - \frac{\varepsilon}{2}\right) (f,f)_\eta,
\end{aligned}
\]
from which the conclusion follows by possibly further decreasing $\eta$ and $\delta$.
\end{proof}

We can now finally conclude the proof:

\begin{proof}[Proof of Theorem~\ref{thm:unif_exp_cv}]
  Fix $f_0,g_0 \in L^2(\nuref)$ such that $f_0,g_0$ both have
  average~1 with respect to $\nuref$ (the general case where $f_0,g_0$ have
  the same average but for a value different from~1 is directly
  obtained by linearity), and set
  $f(t) = \rme^{t \sL_\delta^*} f_0$ and
  $g(t) = \rme^{t \sL_\delta^*} g_0$. 
  Note that $f(t)-g(t)$ has zero
  average with respect to $\nuref$ for all $t \geq 0$, as can be seen
  for instance by computing the time derivative:
\[
  \begin{aligned}
    \frac{d}{dt}\left[\int_\cE (f(t)-g(t)) \, d\nuref\right] & = \int_\cE
    \sL_\delta^*(f(t)-g(t))\, d\nuref\\
    & = \int_\cE (f(t)-g(t))(\sL_\delta
    \bone) \, d\nuref = 0.
  \end{aligned}
\]
Therefore, for the values $\varepsilon, \eta,\delta$ given by Corollary~\ref{corr:global_coercivity},
  \[
  \begin{aligned}
    \frac{d}{dt}\left[\frac12 \Big(f(t)-g(t),f(t)-g(t)\Big)_\eta \right] 
    & = \Big( \sL_\delta^* [f(t)-g(t)], f(t)-g(t)\Big)_\eta \\
      & \leq -(\lambda-\varepsilon)\Big(f(t)-g(t),f(t)-g(t)\Big)_\eta.
  \end{aligned}
  \]
  By a Gronwall inequality,
  \[
  \Big( f(t)-g(t),f(t)-g(t)\Big)_\eta \leq \rme^{-2(\lambda -
    \varepsilon)t} 
  \Big( f_0-g_0,f_0-g_0\Big)_\eta.
\]
This contraction estimate allows to prove that there is a unique limiting function $h_\delta \in L^2(\nuref)$, by following the argument provided in~\cite[Section~3.1]{BHM16}. The exponential contraction~\eqref{eq:exp_contraction_to_h_delta} is then obtained by setting $g_0 = h_\delta$. The constant $c_\varepsilon$ in~\eqref{eq:exp_contraction_to_h_delta} is $\sqrt{(1+\eta)/(1-\eta)}$ by the equivalence of norms~\eqref{eq:scalar_product}. It is smaller than~2 for $\eta \leq 3/5$, which can be assumed without loss of generality.
\end{proof}
  
\subsection{Proof of Lemma~\ref{lem:commutation_B_partial_z}}
\label{sec:commutator}

We show the result for $B\partial_z - \partial_z B$, the proof being
similar for $B^*\partial_z - \partial_z B^*$. Note first that
\[
\begin{aligned}
  \partial_z B & = \partial_z \left(1+\frac{1}{\beta m} \nabla_q^* \nabla_q \right)^{-1} \Pi_p \cL_{\rm ham} \\
  & = \left[\partial_z, \left(1+\frac{1}{\beta m}\nabla_q^* \nabla_q
    \right)^{-1}\right] \Pi_p \cL_{\rm ham} + \left(1+\frac{1}{\beta m}\nabla_q^* \nabla_q \right)^{-1} 
  \partial_z \Pi_p \cL_{\rm ham}.
\end{aligned}
\]
Since $\Pi_p$ and $\partial_z$ commute, 
\[
\partial_z \Pi_p \cL_{\rm ham} = \Pi_p \cL_{\rm ham}\partial_z - \Pi_p \nabla_q(\partial_z \Uk)^T \nabla_p,
\]
so that
\[
[\partial_z,B] = \left[\partial_z, \left(1+\frac{1}{\beta m}\nabla_q^* \nabla_q \right)^{-1}\right] \Pi_p \cL_{\rm ham} - \left(1+\frac{1}{\beta m}\nabla_q^* \nabla_q \right)^{-1}\Pi_p \nabla_q(\partial_z \Uk)^T \nabla_p.
\]

The second operator on the right-hand side is bounded since $\nabla_q(\partial_z U)$ is uniformly bounded and $\Pi_p \nabla_p$ is a bounded operator on $L^2(\nuref)$, as can be seen by writing the action of this operator in a Hermite basis of the momenta variables (as made precise in~\cite{RS17} for instance). To conclude the proof, it remains to show that the first operator on the right-hand side of the last equality is bounded. We use to this end the commutator identity 
\[
\left[\nabla_q^*\nabla_q,\partial_z\right] = \beta \nabla_q(\partial_zU)^T \nabla_q,
\]
so that
\[
\begin{aligned}
& \left[\partial_z, \left(1+\frac{1}{\beta m}\nabla_q^* \nabla_q \right)^{-1}\right] \\
& \quad = \left(1+\frac{1}{\beta m}\nabla_q^* \nabla_q \right)^{-1} \left[1+\frac{1}{\beta m}\nabla_q^* \nabla_q,\partial_z\right] \left(1+\frac{1}{\beta m}\nabla_q^* \nabla_q \right)^{-1} \\
& \quad = \frac{1}{m}\left(1+\frac{1}{\beta m}\nabla_q^* \nabla_q \right)^{-1}\nabla_q(\partial_zU)^T \nabla_q\left(1+\frac{1}{\beta m}\nabla_q^* \nabla_q \right)^{-1}.
\end{aligned}
\]
Therefore,
\[
\left[\partial_z, \left(1+\frac{1}{\beta m}\nabla_q^* \nabla_q \right)^{-1}\right] \Pi_p \cL_{\rm ham} = ST,
\]
with
\[
S = \frac{1}{m}\left(1+\frac{1}{\beta m}\nabla_q^* \nabla_q \right)^{-1}\nabla_q(\partial_zU)^T \nabla_q, \qquad T = \left(1+\frac{1}{\beta m}\nabla_q^* \nabla_q \right)^{-1}\Pi_p\cL_{\rm ham}.
\]
Let us show that both $S$ and $T$ are bounded operators on $L^2(\nuref)$. Since $\nabla_q^* \varphi = - \nabla_q \varphi+ \beta (\nabla U)\varphi$ and $\nabla U, \nabla_q(\partial_zU)$ are uniformly bounded functions, it suffices to prove that the operator 
\[
\widetilde{S} = \left(1+\frac{1}{\beta m}\nabla_q^* \nabla_q \right)^{-1}\nabla_q^*
\]
is bounded in order to conclude that $S$ is bounded. Now, $\widetilde{S} \widetilde{S}^* = F(\nabla_q^* \nabla_q)$ with $F$ bounded, so that $\widetilde{S} \widetilde{S}^*$ is bounded by spectral calculus (see for instance~\cite{ReedSimon1,DautrayLions3}). Therefore,
\[
\left\|\widetilde{S}^* \varphi \right\|_{L^2(\nuref)}^2 = \left\langle \varphi, \widetilde{S} \widetilde{S}^* \varphi \right\rangle_{L^2(\nuref)} \leq \left\| \widetilde{S} \widetilde{S}^* \right\|_{\mathcal{B}(L^2(\nuref))} \|\varphi\|_{L^2(\nuref)}^2.
\]
This shows that $\widetilde{S}^*$ is bounded, hence $\widetilde{S}$ is bounded. For the operator~$T$, we note that (with inequalities in the sense of symmetric operators)
\[
0 \leq -\Pi_p\cL_{\rm ham}^2 \Pi_p = (\Pi_p\cL_{\rm ham})(\Pi_p\cL_{\rm ham})^* = \frac1\beta \nabla_q^* M^{-1} \nabla_q \leq \frac{1}{\beta \overline{m}} \nabla_q^* \nabla_q,
\]
where $M \geq \overline{m} \,\mathrm{Id}$ with $\overline{m} > 0$. Therefore, $TT^*$ is bounded again by spectral calculus, hence $T$ is bounded. 

\subsection{Proof of Corollary~\ref{corr:resolvent_estimates}}
\label{sec:proof_corollary:resolvent_estimates}

Let us first note that~\eqref{eq:exp_contraction_to_h_delta} can be rewritten as 
\[
\begin{aligned}
  & \forall f \in L^2(\nuref), \quad \forall \delta \in
  (0,\delta_\varepsilon], \quad \forall t \geq 0, \\
    & \qquad \qquad \left\|\rme^{t \sL_\delta^*}Q_\delta
    f\right\|_{L^2(\nuref)} \leq c_\varepsilon
    \rme^{-(\lambda-\varepsilon)t}\left\|Q_\delta f
    \right\|_{L^2(\nuref)}.
\end{aligned}
\]
In fact, since $Q_\delta^2 = Q_\delta$ hence $\left\|Q_\delta \right\|_{\mathcal{B}(L^2(\nuref))} \leq 1$,
\begin{equation}
  \label{eq:exp_decay_op_norm}
  \forall \delta \in (0,\delta_\varepsilon], \quad \forall t \geq 0, 
    \qquad 
    \left\|\rme^{t \sL_\delta^*}Q_\delta \right\|_{\mathcal{B}(Q_\delta L^2(\nuref))} \leq c_\varepsilon
    \rme^{-(\lambda-\varepsilon)t}.
\end{equation}
The next point is that $Q_\delta$ and $\rme^{t \sL_\delta^*}$ commute since
\[
  Q_\delta \rme^{t \sL_\delta^*}f = \rme^{\sL_\delta^*}f 
  - \left({\textstyle\int_\cE} \rme^{t \sL_\delta^*}f \, \nuref\right)h_\delta 
  = \rme^{t \sL_\delta^*}f - \left({\textstyle\int_\cE}  f \, \nuref\right)h_\delta 
  = \rme^{t \sL_\delta^*}Q_\delta f.
\]
Therefore, the operator
\[
U_\delta = -\int_0^{\infty} \rme^{t \sL_\delta^*} Q_\delta \, dt
\]
is a well defined bounded operator by~\eqref{eq:exp_decay_op_norm},
and a simple computation shows that
\[
U_\delta \sL_\delta^* = \sL_\delta^* U_\delta = Q_\delta. 
\]
This shows that $\sL_\delta^* Q_\delta$ is invertible on $Q_\delta L^2(\nuref)$ with inverse $U_\delta$, and that
\[
  \begin{aligned}
    \left\| \left(\sL_\delta^* Q_\delta\right)^{-1}
    \right\|_{\mathcal{B}(Q_\delta L^2(\nuref))} & \leq \int_0^{\infty}
    \left\|\rme^{t \sL_\delta^*} Q_\delta
    \right\|_{\mathcal{B}(Q_\delta L^2(\nuref))} \, dt \\
    & \leq c_\varepsilon \int_0^{\infty} \rme^{-(\lambda-\varepsilon)t} \,
    dt,
  \end{aligned}
\]
which leads to the desired operator bound. 

For the estimates on $(\sL_\delta P_\delta)^{-1}$, we note that
$Q_\delta^* = P_\delta$, so that
$(\sL_\delta P_\delta)^* = Q_\delta\sL_\delta^* = \sL_\delta^* Q_\delta$. The bound on
$\left(\sL_\delta^* Q_\delta\right)^{-1}$ therefore immediately
implies the bound on its adjoint
$\left(\sL_\delta P_\delta\right)^{-1}$.

\section{Proof of Theorem~\ref{thm:inv_meas}}
\label{sec:proof_inv_meas}

In order to guarantee that the correction function $\fh$ is in $L^2(\nuref)$, we rely on the regularity of 
\begin{equation}
  \label{eq:def_g1}
  g_1 = \cL_1^* \bone = \left(\frac{\beta}{\overline{\beta}}-1\right) W \in C^\infty(\cD_q\times\cD_z), 
\end{equation}
where we recall $W$ is defined in~\eqref{eq:def_W}; as well as the following lemma, whose proof is postponed till Sec.~\ref{section:commutator2}.

\begin{lemma}
  \label{lem:reg_commutator_dz}
  The operator $\partial_z (\cL_0^*)^{-1}(1-\Pi_z)$ is bounded from
  $H^1(\nuref)$ to $L^2(\nuref)$, and
  $\partial_z^2 (\cL_0^*)^{-1}(1-\Pi_z)$ is bounded from $H^2(\nuref)$
  to $L^2(\nuref)$. Therefore, $\cL_1^* (\cL_0^*)^{-1}$ is bounded
  from $H^2(\nuref)$ to $L^2(\nuref)$.
\end{lemma}

Using this result, the next step is to construct functions
$\fh,\wfh \in L_0^2(\nuref)$ such that
\begin{equation}
  \label{eq:asymptotic_residual_FP}
  \left(\frac1\delta \cL_0 + \cL_1\right)^*\left(1 + \delta \fh 
    + \delta^2 \wfh\right) = \delta^2 \cL_1^* \wfh.
\end{equation}
The equality~\eqref{eq:asymptotic_residual_FP} can be obtained from the hierarchy
\[
\begin{aligned}
  \cL_0^* \fh & = - \cL_1^* \bone = -g_1, \\
  \cL_0^* \wfh & = - \cL_1^* \fh.
\end{aligned}
\]
These equations should be understood as a family of PDEs in~$x$,
parametrized by~$z$. Let us give the formal solutions to these
equations, discussing regularity and integrability issues only in a
second step. For a given value of~$z$, the solvability condition for
the first equation is that $g_1(\cdot,z)$ has zero average with
respect to $\nurefz(dx)$ for all $z \in \cD_z$ (see the estimate recalled in Lemma~\ref{lem:conditioned_coercivity} which implies the invertibility of $\cL_0^*$ on $\mathrm{Ran}(1-\Pi_z)$). 
It is at this stage
that the choice of the reference measure is crucial. In view of~\eqref{eq:def_g1}, it suffices to show that $W$ has zero average with respect to $\nurefz(dx)$ for all $z \in \cD_z$. We rewrite to this end the definition~\eqref{eq:def_W} as
\[
W = -\left[\partial_z^2 \Uk - \beta (\partial_z \Uk)^2 - \Ak'' + \beta \Ak' \partial_z \Uk \right] - (\overline{\beta}-\beta)(\partial_z \Uk - \Ak')\Ak'.
\]
Each of the two factors in the last equality has mean~0 with respect
to $\nurefz(dx)$ in view of the following identities:
\begin{equation}
  \label{eq:derivatives_A_kappa}
  \begin{aligned}
  \Ak'(z) & = \frac{\int_{\cD_q} \partial_z \Uk(q,z) \, \rme^{-\beta
      \Uk(q,z)}\, dq}
  {\int_{\cD_q} \rme^{-\beta \Uk(q,z)}\, dq}, \\
  \Ak''(z) & = \frac{\int_{\cD_q} \left[\partial^2_z \Uk(q,z) - 
      \beta (\partial_z \Uk(q,z))^2\right]\, \rme^{-\beta \Uk(x,z)}\, dq}
  {\int_{\cD_q} \rme^{-\beta \Uk(q,z)}\, dq} + \beta \left(\Ak'(z)\right)^2,
  \end{aligned}
\end{equation}
hence $W$ is of mean~0 with respect to $\nurefz(dx)$. Therefore,
$(\cL_0^*)^{-1} g_1$ is well defined and has zero average with respect
to $\nurefz(dx)$ for any value of~$z \in \cD_z$. The general solution
of $\cL_0^* \mathfrak{h} = -g_1$ is then
\[
\fh(x,z) = -\left(\left(\cL_0^*\right)^{-1} g_1\right)(x,z) + \overline{\fh}(z),
\]
with $\overline{\fh}$ unspecified at this stage.

Let us now prove that it is possible to construct a function $\wfh$
such that $\cL_0^* \wfh = -\cL_1^* \fh$. This condition will in fact
determine $\overline{\fh}$. The solvability condition reads
\[
\Pi_z \cL_1^* \fh = 0 = \Pi_z \cA \Pi_z \overline{\fh} - G, \qquad 
G = \Pi_z \cL_1^* \left(\cL_0^*\right)^{-1} g_1,
\]
where we have used again that $\Pi_z W = 0$ as well as the reformulation $\cA = -\overline{\beta}^{-1} \partial_z^* \partial_z$ of~\eqref{eq:mathcal_A}. Note that $G(z)$ has zero average with respect to $\onuref(dz)$ since
\[
\int_{\cD_z} G \, d\onuref 
= \int_{\cE} \cL_1^* \left(\cL_0^*\right)^{-1} g_1 \, \nuref 
= \int_{\cE} \left[\left(\cL_0^*\right)^{-1} g_1 \right] 
\left(\cL_1 \bone\right) \, \nuref = 0.
\]
Therefore, 
\[
\overline{\fh} = \left(\Pi_z \cA \Pi_z\right)^{-1} G \in L^2_0(\onuref) = \left\{f \in L^2(\onuref), \ \int_{\cD_z} f(z) \, \onuref(dz) = 0 \right\},
\]
since $\Pi_z \cA \Pi_z$ is invertible on $L^2_0(\onuref)$ (from the coercivity property implied by the Poincar\'e inequality~\eqref{eq:Poincare_marginal}). It is then also possible to consider
\[
\wfh = -(\cL_0^*)^{-1} (1-\Pi_z) \cL_1^* \fh.
\]
A simple extension of Lemma~\ref{lem:reg_commutator_dz} 
shows that $\mathfrak{h} \in H^4(\nuref)$ in fact, so that
$\wfh \in H^2(\nuref)$ and $\cL_1^* \wfh \in L^2(\nuref)$.

To conclude the proof, we note that 
\[
\left(\frac1\delta \cL_0 + \cL_1\right)^*\Big(h_\delta 
- (1 + \delta \fh + \delta^2 \wfh)\Big) = -\delta^2 \cL_1^* \wfh.
\]
Now, 
\[
Q_\delta \Big(h_\delta - (1 + \delta \fh + \delta^2 \wfh)\Big) 
= h_\delta - (1 + \delta \fh + \delta^2 \wfh),
\]
so that the desired estimate follows from Corollary~\ref{corr:resolvent_estimates}.

\subsection{Proof of Lemma~\ref{lem:reg_commutator_dz}}
\label{section:commutator2}

  Let us first recall that $\nabla_q \cL_0^{-1}(1-\Pi_z)$ and
  $\nabla_p (\cL_0^*)^{-1}(1-\Pi_z)$ are bounded operators, for the
  overdamped and inertial dynamics respectively. Indeed, for
  overdamped dynamics,
  \[
  \begin{aligned}
  \|\nabla_q \phi \|_{L^2(\nuref)}^2 & = \|\nabla_q (1-\Pi_z) \phi
  \|_{L^2(\nuref)}^2\\
  & = \left\langle \nabla_q^* \nabla_q (1-\Pi_z)\phi, (1-\Pi_z)\phi\right\rangle \\
  & = -\beta \left\langle \cL_0 (1-\Pi_z)\phi, (1-\Pi_z)\phi\right\rangle,
  \end{aligned}
  \]
  so that 
  \[
  \|\nabla_q \phi \|_{L^2(\nuref)}^2 \leq \beta
  \|\cL_0(1-\Pi_z)\phi\|_{L^2(\nuref)} 
  \|(1-\Pi_z)\phi\|_{L^2(\nuref)}.
  \]
  Replacing $(1-\Pi_z)\phi$ by $\cL_0^{-1}(1-\Pi_z)\varphi$ shows that
  $\nabla_q \cL_0^{-1}(1-\Pi_z)$ is bounded on $L^2(\nuref)$, with
  operator norm lower than
  $\sqrt{\beta}
  \|\cL_0^{-1}(1-\Pi_z)\|^{1/2}_{\mathcal{B}(L^2(\nuref))}$. A similar
  computation shows that, for the inertial dynamics,
  $\nabla_p (\cL_0^*)^{-1}(1-\Pi_z)$ is bounded on $L^2(\nuref)$, with
  operator norm also lower than
  $\sqrt{\beta}
  \|\cL_0^{-1}(1-\Pi_z)\|^{1/2}_{\mathcal{B}(L^2(\nuref))}$.
  
  The result now relies on commutator identities. For the overdamped
  dynamics, $\cL_0^* = \cL_0$ and
  \[
  [\partial_z, \cL_0(1-\Pi_z)] = -\nabla_q (\partial_z \Uk)^T \nabla_q,
  \]
  so that
  \begin{equation}
    \label{eq:commutator_dz_overdamped}
    \partial_z (\cL_0^*)^{-1} (1-\Pi_z) - (\cL_0^*)^{-1}
    (1-\Pi_z) \partial_z 
    = \cL_0^{-1} (1-\Pi_z) \nabla_q (\partial_z \Uk)^T \nabla_q \cL_0^{-1}(1-\Pi_z).
  \end{equation}
  It is then easy to see that
  $\partial_z (\cL_0^*)^{-1} (1-\Pi_z) - (\cL_0^*)^{-1}
  (1-\Pi_z) \partial_z$ is bounded on $L^2(\nuref)$ since
  $\nabla_q \cL_0^{-1}(1-\Pi_z)$ is bounded on
  $L^2(\nuref)$. Therefore, $\partial_z (\cL_0^*)^{-1} (1-\Pi_z)$ is
  bounded from $H^1(\nuref)$ to $L^2(\nuref)$. For the inertial
  dynamics, the proof follows the same lines since
  \[
  [\partial_z, \cL_0^*(1-\Pi_z)] = \nabla_q (\partial_z \Uk)^T \nabla_p.
  \]
  
  For second order derivatives in~$z$, we again use the above
  commutation rule. For the overdamped dynamics, a repeated use of~\eqref{eq:commutator_dz_overdamped} 
  gives
  \[
  \begin{aligned}
  & \partial_z^2 (\cL_0^*)^{-1} (1-\Pi_z)\\
  & = \partial_z \cL_0^{-1} (1-\Pi_z) \partial_z 
  + \partial_z \cL_0^{-1} (1-\Pi_z) \nabla_q (\partial_z \Uk)^T \nabla_q \cL_0^{-1}(1-\Pi_z).
  \\
  & = \cL_0^{-1} (1-\Pi_z) \partial_z^2 + \cL_0^{-1} (1-\Pi_z) 
  \nabla_q (\partial_z \Uk)^T \nabla_q \cL_0^{-1}(1-\Pi_z)\partial_z \\
  & \quad+ \cL_0^{-1} (1-\Pi_z) \partial_z 
  \left[\nabla_q (\partial_z \Uk)^T \nabla_q\right] \cL_0^{-1}(1-\Pi_z) \\
& \quad + \cL_0^{-1} (1-\Pi_z) \nabla_q (\partial_z \Uk)^T \nabla_q 
\cL_0^{-1}(1-\Pi_z)\nabla_q (\partial_z \Uk)^T \nabla_q \cL_0^{-1}(1-\Pi_z).
  \end{aligned}
  \]
  The first two operators on the right-hand side of the above equality are clearly bounded from $H^2(\nuref)$ to $L^2(\nuref)$, while the last one is even bounded on $L^2(\nuref)$. For the remaining operator, we note that
\[
\begin{aligned}
& \cL_0^{-1} (1-\Pi_z) \partial_z \left[\nabla_q (\partial_z \Uk)^T
  \nabla_q\right] 
\cL_0^{-1}(1-\Pi_z)\\
& \qquad = \cL_0^{-1} (1-\Pi_z) \nabla_q (\partial^2_z \Uk)^T \nabla_q \cL_0^{-1}(1-\Pi_z) \\
& \qquad \quad + \cL_0^{-1} (1-\Pi_z) \nabla_q (\partial_z \Uk)^T \nabla_q 
\partial_z \cL_0^{-1}(1-\Pi_z), 
\end{aligned}
\]
which is the sum of a bounded operator on $L^2(\nuref)$, and the
composition of the bounded operator
$\cL_0^{-1} (1-\Pi_z) \nabla_q (\partial_z \Uk)^T \nabla_q$ on
$L^2(\nuref)$ (the argument for that being similar to the argument
showing that $\nabla_q \cL_0^{-1}$ is bounded) 
and the operator
$\partial_z \cL_0^{-1}(1-\Pi_z)$, which is bounded from $H^1(\nuref)$
to $L^2(\nuref)$. This allows to conclude that
$\partial_z^2 (\cL_0^*)^{-1} (1-\Pi_z)$ is bounded from $H^2(\nuref)$
to $L^2(\nuref)$ for the overdamped dynamics. Similar manipulations
give the same result for the inertial dynamics.

\section{Proof of Proposition~\ref{prop:approx_solution_Poisson}}
\label{sec:proof_expansion}

Fix $\varphi \in H^2(\nuref)$. We seek to solve
$-\cL_\delta \Phi_\delta = P_\delta \varphi$ by approximating the
solution by $\Psi + \delta \psi$. If we can construct
$\Psi,\psi \in H^2(\nuref)$ such that
\[
\left\{ \begin{aligned}
-\cL_0 \Psi & = 0, \\
-\cL_0 \psi & = \cL_1 \Psi + P_0\varphi, 
\end{aligned} \right.
\]
then
\[
-\sL_\delta(\Phi_\delta - \Psi - \delta \psi) = \delta \cL_1 \psi + (P_\delta-P_0)\varphi. 
\]
If $\cL_1\psi \in L^2(\nuref)$ (which is indeed the case when
$\psi \in H^2(\nuref)$) with, for some $K > 0$ independent of
$\varphi$,
\begin{equation}
  \label{eq:estimate_L1_psi}
  \left\|\cL_1\psi\right\|_{L^2(\nuref)} \leq K \|\varphi\|_{H^2(\nuref)},
\end{equation}
then, by construction, the right-hand side of the above equation is in
$P_\delta L^2(\nuref)$ since it is in the image of $\sL_\delta$ and
hence has zero average with respect to~$\nu_\delta$. 
Therefore, the desired estimate directly follows
from~\eqref{eq:diff_Pdelta_P0} and the resolvent estimates provided by
Corollary~\ref{corr:resolvent_estimates}.

Let us now turn to the construction of the functions $\Psi,\psi$,
following a strategy very similar to the one used for the proof of
Theorem~\ref{thm:inv_meas}. For the regularity estimates, we use the
fact that $\cL_1$ is a bounded operator from $H^{n+2}(\nuref)$ to
$H^n(\nuref)$ for any $n \geq 0$, and that $\cA^{-1}$ is a bounded
operator from $H^n(\onuref) \cap L^2_0(\onuref)$ to $H^{n+2}(\onuref)$
for any $n \geq 0$. The equation $-\cL_0 \Psi = 0$ shows that
$\Psi(x,z) = \overline{\Psi}(z)$ is a function of~$z$ only. The
solvability condition for $-\cL_0 \psi = \cL_1 \Psi + P_0\varphi$ is
\[
\Pi_z\left(\cL_1 \Psi + P_0\varphi \right) = 0 
= \cA \overline{\Psi} + \Pi_z P_0 \varphi.
\]
Since $\Pi_z P_0 \varphi$ has zero average with respect to $\onuref$,
this shows that
\[
\overline{\Psi} = -\cA^{-1}\Pi_z P_0 \varphi.
\]
It is clear, from the regularizing properties of $\cA^{-1}$, that
$\overline{\Psi} \in H^4(\onuref)$ since $P_0 \varphi \in H^2(\onuref)$. Moreover, it is possible in this
case to define
\[
\psi = \cL_0^{-1}(1-\Pi_z)\left( \cL_1 \Pi_z \cA^{-1}\Pi_z - 1\right)P_0\varphi.
\]
It is also easy to check that $\psi \in H^2(\nuref)$ and
that~\eqref{eq:estimate_L1_psi} is satisfied.


\subsection*{Acknowledgements}
Part of this work was done during the authors' stay at the Institut
Henri Poincar\'e - Centre Emile Borel during the trimester
``Stochastic Dynamics Out of Equilibrium'' (April-July 2017). The authors warmly thank
this institution for its hospitality, and Inria Paris who funded the stay
of EVE through an invited professor fellowship. The work of GS was
funded in part by the Agence Nationale de la Recherche, under grant
ANR-14-CE23-0012 (COSMOS) and by the European Research Council under
the European Union's Seventh Framework Programme (FP/2007-2013) / ERC
Grant Agreement number 614492. He also benefited from the scientific
environment of the Laboratoire International Associ\'e between the
Centre National de la Recherche Scientifique and the University of
Illinois at Urbana-Champaign. The work of EVE was funded in part by
the Materials Research Science and Engineering Center (MRSEC) program
of the National Science Foundation (NSF) under award number
DMR-1420073 and by NSF under award number DMS-1522767.


\end{document}